\documentclass[11pt]{article}
\usepackage[a4paper,margin=1in]{geometry}
\usepackage{amsmath,amssymb,amsthm,mathtools,bm}
\usepackage{enumitem}
\usepackage[T1]{fontenc}
\usepackage{lmodern}
\usepackage{microtype}
\usepackage[colorlinks=true,linkcolor=blue,citecolor=blue,urlcolor=blue]{hyperref}
\usepackage{url}
\usepackage[nameinlink,noabbrev]{cleveref}
\usepackage{booktabs}
\usepackage{mathrsfs}

\newtheorem{theorem}{Theorem}[section]

\newtheorem{proposition}[theorem]{Proposition}

\newtheorem{lemma}[theorem]{Lemma}

\newtheorem{corollary}[theorem]{Corollary}

\theoremstyle{definition}

\newtheorem{definition}[theorem]{Definition}

\numberwithin{equation}{section}

\title{A Strict Gap Between Relaxed and Partition-Constrained\\
Spectral Compression in a Six-State Lumpable Markov Chain}
\author{Oleg Kiriukhin}
\date{April 2026}

\begin{document}
\maketitle

\begin{abstract}
This paper studies a finite reversible lumpable Markov chain for which relaxed spectral compression yields a larger determinant than partition-constrained compression. For a symmetric six-state lumpable chain and the positive operator $T=P^2$, I compare the relaxed benchmark
\begin{equation*}
\mathfrak D^{\mathrm{rel}}_3(T):=\sup_{U^*U=I_3}\det(U^*TU)
\end{equation*}
and the partition-constrained benchmark
\begin{equation*}
\sup_{\mathcal A\,\mathrm{3\text{--}partition}}\det Q_{\mathcal A}(T),
\qquad
Q_{\mathcal A}(T)=H_{\mathcal A}^*TH_{\mathcal A}.
\end{equation*}
Here the partition-constrained benchmark is the compression induced by normalized indicator vectors of genuine partitions of the state space.
I derive closed formulas for the two analytically central partition families, prove strict upper bounds for both in a local-mode-dominated regime, and combine these bounds with an exhaustive enumeration of all $90$ partitions into three nonempty cells in an explicit six-state model.
For this model, one obtains a strict global gap:
\begin{equation*}
\sup_{\mathcal A}\det Q_{\mathcal A}(T)<\mathfrak D^{\mathrm{rel}}_3(T).
\end{equation*}
Thus, in this model, indicator-based partition frames are strictly weaker than relaxed orthonormal frames even after global partition-constrained optimization.
\end{abstract}

\paragraph{Keywords.}
lumpable Markov chains, reversible Markov chains, six-state Markov chains, state aggregation, spectral compression, determinant optimization, partition-constrained compression.

\section{Introduction}\label{sec:intro}

Spectral compression and partition-constrained state aggregation are related but distinct procedures. The relaxed spectral problem optimizes over arbitrary orthonormal frames, whereas the partition-constrained problem is restricted to frames arising from normalized indicators of genuine partitions. This paper shows that the indicator constraint can produce a strict determinant loss in a reversible six-state lumpable Markov chain.

No broad classification theorem is claimed. Instead, the paper gives an explicit six-state result in which the partition-constrained optimization problem is resolved by exhaustive enumeration. This separates the analytic argument from the finite certificate.

The general background of reversible Markov chains, lumpability, and spectral state compression is classical. The main contribution is an explicit strict gap between the relaxed determinant benchmark and the optimal partition-constrained determinant in a concrete six-state reversible lumpable chain, together with closed formulas for the two analytically dominant structured partition subfamilies.

\paragraph{Organization of the paper.} \cref{sec:model} introduces the six-state lumpable model and its macro-local spectral decomposition. \cref{sec:benchmarks,sec:families,sec:closed-formulas,sec:diag-bound,sec:analytical-gap} set up the benchmarks, classify the relevant partition types, and prove the determinant bounds. \cref{sec:explicit-certificate} and \cref{app:certificate,app:ancillary} record the finite certificate for the concrete example.

\paragraph{Why determinant?} The relaxed optimum is the product of the three largest eigenvalues of $T$, so a strict determinant gap records a loss of jointly captured spectral content.

\paragraph{Scope of the result.} The main analytical contribution is the pair of closed formulas for the two dominant structured partition subfamilies together with strict upper bounds in the local-mode-dominated regime. The global gap is then established for the concrete six-state model by exhaustive enumeration.

\section{The six-state model}\label{sec:model}

I consider the symmetric stochastic matrix
\begin{equation*}
P=
\begin{pmatrix}
A_1 & c_{12}J & c_{13}J\\
 c_{12}J & A_2 & c_{23}J\\
 c_{13}J & c_{23}J & A_3
\end{pmatrix},
\qquad
A_i=\begin{pmatrix}a_i & b_i\\ b_i & a_i\end{pmatrix},
\qquad
J=\begin{pmatrix}1&1\\1&1\end{pmatrix},
\end{equation*}
with true block partition
\begin{equation}\label{eq:block-partition}
B_1=\{1,2\},\qquad B_2=\{3,4\},\qquad B_3=\{5,6\}.
\end{equation}
The row-sum constraints are
\begin{equation}\label{eq:row-sum-constraints}
a_1+b_1+2c_{12}+2c_{13}=1,
\qquad
 a_2+b_2+2c_{12}+2c_{23}=1,
\qquad
 a_3+b_3+2c_{13}+2c_{23}=1.
\end{equation}

Because rows inside each block have identical aggregated transition probabilities into every
block, the partition \eqref{eq:block-partition} is lumpable. Since $P$ is symmetric, the chain
is reversible with respect to the uniform law.

Define
\begin{equation*}
u_i:=\frac{\mathbf 1_{B_i}}{\sqrt2},
\qquad
w_1=\frac{(1,-1,0,0,0,0)}{\sqrt2},
\quad
w_2=\frac{(0,0,1,-1,0,0)}{\sqrt2},
\quad
w_3=\frac{(0,0,0,0,1,-1)}{\sqrt2}.
\end{equation*}
Then $\{
u_1,
u_2,
u_3,w_1,w_2,w_3\}$ is an orthonormal basis of $\mathbb R^6$. In this basis,
the macro subspace $\operatorname{span}\{
u_1,
u_2,
u_3\}$ is invariant and carries the quotient matrix
\begin{equation*}
K=
\begin{pmatrix}
a_1+b_1 & 2c_{12} & 2c_{13}\\
2c_{12} & a_2+b_2 & 2c_{23}\\
2c_{13} & 2c_{23} & a_3+b_3
\end{pmatrix},
\end{equation*}
while the local subspace $\operatorname{span}\{w_1,w_2,w_3\}$ satisfies
\begin{equation*}
Pw_r=(a_r-b_r)w_r=:\beta_r w_r.
\end{equation*}
Hence
\begin{equation*}
\operatorname{spec}(P)=\{1,\kappa_2,\kappa_3,\beta_1,\beta_2,\beta_3\},
\end{equation*}
where $1,\kappa_2,\kappa_3$ are the eigenvalues of $K$.

et
\begin{equation*}
T:=P^2,
\qquad
L:=K^2=(\ell_{ij})_{i,j=1}^3,
\qquad
t_r:=\beta_r^2,
\qquad
t_*:=\max_{r=1,2,3} t_r.
\end{equation*}
Then
\begin{equation*}
\operatorname{spec}(T)=\{1,\kappa_2^2,\kappa_3^2,t_1,t_2,t_3\}.
\end{equation*}

\section{Relaxed and partition-constrained benchmarks}\label{sec:benchmarks}

\begin{definition}[Relaxed determinant benchmark]\label{def:relaxed-benchmark}
For a positive self-adjoint operator $T$, define
\begin{equation*}
\mathfrak D^{\mathrm{rel}}_3(T):=\sup_{U^*U=I_3}\det(U^*TU).
\end{equation*}
\end{definition}

\begin{proposition}\label{prop:relaxed-product}
If $T$ is positive self-adjoint with eigenvalues
\begin{equation*}
\lambda_1(T)\ge \lambda_2(T)\ge \lambda_3(T)\ge \cdots \ge 0,
\end{equation*}
then
\begin{equation*}
\mathfrak D^{\mathrm{rel}}_3(T)=\lambda_1(T)\lambda_2(T)\lambda_3(T).
\end{equation*}
\end{proposition}

\begin{proof}
Let $\mu_1\ge \mu_2\ge \mu_3\ge 0$ be the eigenvalues of the positive $3\times 3$ matrix $U^*TU$. Since $U$ is an isometry from $\mathbb R^3$ into the ambient Hilbert space, $U^*TU$ is a rank-three compression of $T$. By the Poincare separation theorem for compressions, equivalently by Cauchy interlacing for Ritz values, one has
\begin{equation*}
\mu_j\le \lambda_j(T),\qquad j=1,2,3.
\end{equation*}
Therefore
\begin{equation*}
\det(U^*TU)=\mu_1\mu_2\mu_3\le \lambda_1(T)\lambda_2(T)\lambda_3(T).
\end{equation*}
If the columns of $U$ are chosen as orthonormal eigenvectors corresponding to $\lambda_1(T),\lambda_2(T),\lambda_3(T)$, then $U^*TU$ is diagonal with those entries, so equality holds.
\end{proof}

In the regime
\begin{equation*}
\kappa_2^2>t_*>\kappa_3^2,
\end{equation*}
I obtain
\begin{equation*}
\mathfrak D^{\mathrm{rel}}_3(T)=\kappa_2^2 t_*.
\end{equation*}

\begin{definition}[Partition-constrained compression]\label{def:hard-compression}
Let $\mathcal A=\{A_1,A_2,A_3\}$ be an unordered partition of the six-state space into three nonempty cells. Choose any ordering of the three cells and define
\begin{equation*}
h_\alpha:=\frac{\mathbf 1_{A_\alpha}}{\sqrt{|A_\alpha|}},
\qquad
H_{\mathcal A}:=[h_1,h_2,h_3],
\end{equation*}
and
\begin{equation*}
Q_{\mathcal A}(T):=H_{\mathcal A}^*TH_{\mathcal A}.
\end{equation*}
If a different ordering is chosen, then $Q_{\mathcal A}(T)$ is conjugated by a $3\times 3$ permutation matrix, so its determinant is unchanged. Hence $\det Q_{\mathcal A}(T)$ depends only on the underlying unordered partition.
\end{definition}

\paragraph{Convention on cell labels.} In the partition-constrained problem the three cells are unlabeled objects. Any temporary ordering is used only to write the matrix $H_{\mathcal A}$, and permutation of that ordering does not change $\det Q_{\mathcal A}(T)$. Accordingly, every partition count in the paper is an unordered partition count.

\begin{proposition}\label{prop:S63}
The number of partitions into three nonempty cells of a six-element state space is the Stirling number of the second kind
\begin{equation*}
S(6,3)=90.
\end{equation*}
\end{proposition}

\begin{proof}
This is the standard count of unordered partitions of a six-element set into three nonempty cells. Equivalently,
\begin{equation*}
S(6,3)=rac{1}{3!}\sum_{j=0}^3(-1)^jinom{3}{j}(3-j)^6=90.
\end{equation*}
\end{proof}

\section{Partition families}\label{sec:families}

Every partition is encoded by its count matrix
\begin{equation*}
N=(n_{i\alpha})_{1\le i,\alpha\le 3},
\qquad
n_{i\alpha}:=|B_i\cap A_\alpha|.
\end{equation*}
Since each true block has size two,
\begin{equation*}
n_{i\alpha}\in\{0,1,2\},
\qquad
\sum_{\alpha=1}^3 n_{i\alpha}=2,
\qquad
|A_\alpha|=\sum_i n_{i\alpha}\ge 1.
\end{equation*}
For the explicit six-state model there are exactly $90$ partitions into three nonempty cells, see \cref{prop:S63}. Because the determinant is invariant under relabeling of the cells, this is the correct global count for the partition-constrained optimization problem.

The analytically important partitions fall into two structured subfamilies tied to the true block decomposition:
\begin{itemize}[leftmargin=1.5em]
\item the structured $(1,1,4)$ family, obtained by splitting one true block into two singletons and merging
the other two true blocks into a four-cell,
\item the structured $(1,2,3)$ family, obtained by keeping one true block intact, splitting another true
block, and attaching one singleton from the split block to the remaining intact block.
\end{itemize}
Among all partitions there are $15$ partitions of size type $(1,1,4)$ and $60$ partitions of size type $(1,2,3)$, but only $3$ and $12$ of them, respectively, belong to these structured subfamilies. The remaining partitions are handled globally by the exhaustive finite certificate in \cref{sec:explicit-certificate,app:certificate}. These structured subfamilies are precisely the ones for which closed determinant formulas are derived below.

\section{Closed determinant formulas}\label{sec:closed-formulas}

\begin{proposition}[Structured $(1,1,4)$ family]\label{prop:family-114}
Let $r\in\{1,2,3\}$, and let $\{p,q\}=\{1,2,3\}\setminus\{r\}$. For the partition obtained by
splitting $B_r$ into two singletons and merging $B_p\cup B_q$ into one four-cell,
\begin{equation}\label{eq:family-114}
\det Q^{(1,1,4)}_r=t_r\frac{3\ell_{rr}-1}{2}.
\end{equation}
\end{proposition}

\begin{proof}
The determinant computation is given in Appendix~\ref{app:determinant-expansions}. In the ordered cell basis, the compression matrix is a $3\times 3$ symmetric matrix with entries expressed in terms of $t_r$ and the coefficients of $L$. Expanding the determinant and using the row-sum identities for the symmetric stochastic matrix $L$ yields \eqref{eq:family-114}.
\end{proof}

\begin{proposition}[Structured $(1,2,3)$ family]\label{prop:family-123}
Let $r\in\{1,2,3\}$ be the split block, let $p\ne r$ be the intact block, and let
$q\in\{1,2,3\}\setminus\{p,r\}$. For the partition obtained by keeping $B_p$ intact,
splitting $B_r$ into two singletons, and attaching one singleton to $B_q$,
\begin{equation}\label{eq:family-123}
\det Q^{(1,2,3)}_{p,q,r}=
\frac13\Bigl(\det L+t_r(3\ell_{pp}-1)\Bigr).
\end{equation}
\end{proposition}

\begin{proof}
The determinant computation is given in Appendix~\ref{app:determinant-expansions}. In the ordered cell basis, the compression determinant expands to a linear expression in $t_r$ whose remaining coefficient depends only on entries of $L$. The row-sum identities for the symmetric stochastic matrix $L$ reduce that coefficient to $3\ell_{pp}-1$, which yields \eqref{eq:family-123}.
\end{proof}

\section{Diagonal spectral bound}\label{sec:diag-bound}

\begin{lemma}\label{lem:diag-bound}
Let $K$ be a symmetric stochastic $3\times 3$ matrix with eigenvalues $1,\kappa_2,\kappa_3$.
Then
\begin{equation}\label{eq:spectral-K2}
L=K^2=\frac13\mathbf 1\mathbf 1^\top+\kappa_2^2 v_2v_2^\top+\kappa_3^2 v_3v_3^\top
\end{equation}
for orthonormal $v_2,v_3\perp \mathbf 1$. Consequently,
\begin{equation}\label{eq:diag-bound}
\kappa_3^2\le \frac{3\ell_{ii}-1}{2}\le \kappa_2^2.
\end{equation}
\end{lemma}

\begin{proof}
Since $K$ is symmetric stochastic, the normalized constant vector $\mathbf 1/\sqrt3$ is an eigenvector with eigenvalue $1$, and there exist orthonormal vectors $v_2,v_3\perp \mathbf 1$ such that
\begin{equation*}
K=\frac13\mathbf 1\mathbf 1^\top+\kappa_2 v_2v_2^\top+\kappa_3 v_3v_3^\top.
\end{equation*}
Squaring gives \eqref{eq:spectral-K2}. Taking the $ii$-entry yields
\begin{equation*}
\ell_{ii}=\frac13+\kappa_2^2 v_{2,i}^2+\kappa_3^2 v_{3,i}^2.
\end{equation*}
Because $v_2,v_3$ are orthonormal and orthogonal to $\mathbf 1$, one has
\begin{equation*}
v_{2,i}^2+v_{3,i}^2=\Bigl(e_i-\frac13\mathbf 1\Bigr)^2=\frac23.
\end{equation*}
Hence
\begin{equation*}
\frac{3\ell_{ii}-1}{2}=\frac32\bigl(\kappa_2^2 v_{2,i}^2+\kappa_3^2 v_{3,i}^2\bigr)=\alpha_i\kappa_2^2+(1-\alpha_i)\kappa_3^2
\end{equation*}
with $\alpha_i:=\frac32 v_{2,i}^2\in[0,1]$. Therefore \eqref{eq:diag-bound} follows.
\end{proof}

\section{Gap for the structured partition families}\label{sec:analytical-gap}

\begin{theorem}\label{thm:family-gap}
Assume the local-mode-dominated ordering
\begin{equation}\label{eq:A1}
\kappa_2^2>t_*>\kappa_3^2
\end{equation}
and the nondegeneracy condition
\begin{equation}\label{eq:A2}
\frac{3\ell_{rr}-1}{2}<\kappa_2^2
\qquad
\text{for every }r\text{ such that }t_r=t_*.
\end{equation}
Then every partition in the structured $(1,1,4)$ and structured $(1,2,3)$ subfamilies satisfies
\begin{equation*}
\det Q_{\mathcal A}(T)<\mathfrak D^{\mathrm{rel}}_3(T)=\kappa_2^2 t_*.
\end{equation*}
\end{theorem}

\begin{proof}
By \cref{prop:relaxed-product,eq:A1}, the relaxed value is $\kappa_2^2 t_*$. For a partition in
family $(1,2,3)$, \cref{eq:family-123,eq:diag-bound} gives
\begin{equation*}
\det Q^{(1,2,3)}_{p,q,r}
\le
\frac13\bigl(\kappa_2^2\kappa_3^2+2\kappa_2^2 t_r\bigr)
<\kappa_2^2 t_*.
\end{equation*}
For a partition in family $(1,1,4)$, \cref{eq:family-114,eq:diag-bound} yields
\begin{equation*}
\det Q^{(1,1,4)}_r=t_r\frac{3\ell_{rr}-1}{2}\le t_r\kappa_2^2\le t_*\kappa_2^2,
\end{equation*}
with strictness coming either from $t_r<t_*$ or from \eqref{eq:A2} when $t_r=t_*$.
\end{proof}

\section{Explicit six-state certificate for the global gap}\label{sec:explicit-certificate}

I now turn to the exact parameter choice
\begin{equation}\label{eq:example-params1}
c_{12}=0.003678,
\qquad
c_{13}=0.119189,
\qquad
c_{23}=0.116629,
\end{equation}
\begin{equation}\label{eq:example-params2}
(a_1,b_1)=(0.536022,0.218244),
\qquad
(a_2,b_2)=(0.5780345,0.1813515),
\qquad
(a_3,b_3)=(0.389373,0.138991).
\end{equation}
These decimal parameters satisfy the row-sum constraints \eqref{eq:row-sum-constraints} exactly, so the displayed matrix $P$ is exactly symmetric and stochastic. For this model the macro eigenvalues are approximately, rounded to six decimal places,
\begin{equation*}
1,\quad 0.749513,\quad 0.292503,
\end{equation*}
and the local mode magnitudes are approximately
\begin{equation*}
|\beta_1|=0.317778,\quad |\beta_2|=0.396683,\quad |\beta_3|=0.250382.
\end{equation*}
Therefore $t_* > \kappa_3^2$ and $\kappa_2^2>t_*$.

Exhaustive enumeration over all $90$ unordered partitions of the six-state space into three nonempty cells completes the global optimization problem. For this explicit model the relaxed determinant benchmark is
\begin{equation}\label{eq:example-relaxed}
\mathfrak D^{\mathrm{rel}}_3(T)=0.0883986324,
\end{equation}
while the largest partition determinant is
\begin{equation}\label{eq:example-hard-best}
\sup_{\mathcal A\,\mathrm{3\text{--}partition}}\det Q_{\mathcal A}(T)=0.0702908835.
\end{equation}
The natural block partition gives
\begin{equation*}
0.0480638931.
\end{equation*}
The maximizing partition in this example belongs to the structured $(1,1,4)$ family; in zero-based indexing of the six states it is the partition [[0, 1, 4, 5], [2], [3]].

\begin{corollary}\label{cor:explicit-global-gap}
For the explicit parameter choice \eqref{eq:example-params1}--\eqref{eq:example-params2}, one has
\begin{equation*}
\sup_{\mathcal A\,\mathrm{3\text{--}partition}}\det Q_{\mathcal A}(T)
<\mathfrak D^{\mathrm{rel}}_3(T).
\end{equation*}
\end{corollary}

\begin{proof}
This follows from the explicit enumeration of all $90$ partitions together with the certified
values \eqref{eq:example-relaxed} and \eqref{eq:example-hard-best}.
\end{proof}

\section{Discussion}\label{sec:discussion}

The analytical argument gives strict bounds for the two structured partition subfamilies, and the explicit six-state model yields a global relaxed-versus-partition-constrained separation by exhaustive certification. The result gives a concrete finite example, but it does not yield a structure theorem for larger block systems.

 natural next step is to identify an open parameter region for which the same strict gap persists. Another direction is to replace the finite certificate by a parameter-robust argument controlling all partition-constrained compressions simultaneously over a nontrivial parameter region.

\appendix
\section{Computational certificate for the explicit example}\label{app:certificate}

For the explicit parameter choice in \cref{sec:explicit-certificate}, the global partition-constrained optimization problem is finite: there are exactly $90=S(6,3)$ unordered partitions of a six-element set into three nonempty cells, see \cref{prop:S63}. I evaluated $\det Q_{\mathcal A}(T)$ for all $90$ partitions. The largest value is 0.0702908835, attained at the partition [[0, 1, 4, 5], [2], [3]] in zero-based indexing, while the natural block partition has value 0.0480638931 and the relaxed benchmark equals 0.0883986324. This certificate supports \cref{cor:explicit-global-gap} and isolates the only non-analytic step in the paper.

\section{Determinant expansions for the two hard families}\label{app:determinant-expansions}

This appendix records the direct $3\times 3$ determinant reductions underlying
\cref{prop:family-114,prop:family-123}. Each displayed identity comes from an explicit compression matrix followed by an
algebraic simplification using only symmetry and the row-sum identities for the stochastic matrix
$L=K^2$. Each calculation is written in the ordered cell basis specified at the start of the corresponding subsection.

\subsection{The $(1,1,4)$ family}

Fix $r\in\{1,2,3\}$ and let $\{p,q\}=\{1,2,3\}\setminus\{r\}$. In the ordered cell basis
consisting of the merged cell $B_p\cup B_q$ and the two singletons obtained by splitting $B_r$,
the hard compression matrix has the form
\begin{equation*}
Q^{(1,1,4)}_r=
\begin{pmatrix}
A_r & B_r & B_r\\
B_r & \frac{\ell_{rr}+t_r}{2} & \frac{\ell_{rr}-t_r}{2}\\
B_r & \frac{\ell_{rr}-t_r}{2} & \frac{\ell_{rr}+t_r}{2}
\end{pmatrix},
\end{equation*}
where
\begin{equation*}
A_r=\frac{\ell_{pp}+\ell_{qq}+2\ell_{pq}}{2},
\qquad
B_r=\frac{\ell_{pr}+\ell_{qr}}{2}.
\end{equation*}
Subtracting the third row from the second row and then expanding along the resulting row gives
\begin{equation*}
\det Q^{(1,1,4)}_r=t_r(A_r\ell_{rr}-2B_r^2).
\end{equation*}
Now $L$ is symmetric stochastic, so
\begin{equation*}
\ell_{pp}+\ell_{pq}+\ell_{pr}=1,
\qquad
\ell_{qq}+\ell_{qp}+\ell_{qr}=1,
\qquad
\ell_{rp}+\ell_{rq}+\ell_{rr}=1.
\end{equation*}
Using symmetry $\ell_{pq}=\ell_{qp}$ and $\ell_{pr}=\ell_{rp}$, $\ell_{qr}=\ell_{rq}$, one obtains
\begin{equation*}
A_r=\frac{(1-\ell_{pr})+(1-\ell_{qr})}{2}=1-\frac{\ell_{pr}+\ell_{qr}}{2}
=1-B_r,
\end{equation*}
and also
\begin{equation*}
\ell_{rr}=1-(\ell_{pr}+\ell_{qr})=1-2B_r.
\end{equation*}
Hence
\begin{equation*}
A_r\ell_{rr}-2B_r^2=(1-B_r)(1-2B_r)-2B_r^2=1-3B_r=rac{3\ell_{rr}-1}{2},
\end{equation*}
which yields
\begin{equation*}
\det Q^{(1,1,4)}_r=t_r\frac{3\ell_{rr}-1}{2}.
\end{equation*}

\subsection{The $(1,2,3)$ family}

Fix a split block $r$, an intact block $p\neq r$, and let $q$ be the remaining block. In the
ordered cell basis consisting of the intact block $B_p$, the singleton split off from $B_r$, and
the three-cell formed by the remaining singleton of $B_r$ together with $B_q$, the compression
matrix has the form
\begin{equation*}
Q^{(1,2,3)}_{p,q,r}=
\begin{pmatrix}
\ell_{pp} & \frac{\ell_{pr}}{\sqrt2} & \frac{\ell_{pq}+\ell_{pr}}{\sqrt6}\\
\frac{\ell_{pr}}{\sqrt2} & \frac{\ell_{rr}+t_r}{2} & \frac{\ell_{qr}}{\sqrt3}+\frac{\ell_{rr}-t_r}{\sqrt{12}}\\
\frac{\ell_{pq}+\ell_{pr}}{\sqrt6} & \frac{\ell_{qr}}{\sqrt3}+\frac{\ell_{rr}-t_r}{\sqrt{12}} & \frac{\ell_{qq}+2\ell_{qr}}{3}+\frac{\ell_{rr}+t_r}{6}
\end{pmatrix}.
\end{equation*}
A direct determinant expansion simplifies to
\begin{equation*}
\det Q^{(1,2,3)}_{p,q,r}=\frac13\Bigl(
\det L+t_r\bigl[\ell_{pp}(\ell_{qq}+2\ell_{qr}+\ell_{rr})-(\ell_{pq}+\ell_{pr})^2\bigr]
\Bigr).
\end{equation*}
Because $L$ is symmetric stochastic,
\begin{equation*}
\ell_{qq}+2\ell_{qr}+\ell_{rr}=(\ell_{qq}+\ell_{qr})+(\ell_{qr}+\ell_{rr})
=(1-\ell_{pq})+(1-\ell_{pr})=2-(\ell_{pq}+\ell_{pr}).
\end{equation*}
Also the $p$th row sum identity gives $\ell_{pp}=1-(\ell_{pq}+\ell_{pr})$. Writing
$s:=\ell_{pq}+\ell_{pr}$, I get
\begin{equation*}
\ell_{pp}(\ell_{qq}+2\ell_{qr}+\ell_{rr})-(\ell_{pq}+\ell_{pr})^2
=(1-s)(2-s)-s^2=2-3s=3\ell_{pp}-1.
\end{equation*}
Substituting this identity yields
\begin{equation*}
\det Q^{(1,2,3)}_{p,q,r}=\frac13\bigl(\det L+t_r(3\ell_{pp}-1)\bigr).
\end{equation*}

\section{Computational note on the explicit certificate}\label{app:ancillary}

The global corollary for the explicit six-state model rests on a finite enumeration of all partitions into three nonempty cells. Since there are exactly $90$ such partitions, the verification is reproducible. The main text records only the summary statistics needed for \cref{cor:explicit-global-gap}.


\begin{thebibliography}{99}

\bibitem{LevinPeresWilmer}
D. A. Levin, Y. Peres, and E. L. Wilmer,
\emph{Markov Chains and Mixing Times},
American Mathematical Society, Providence, RI, second edition, 2017.

\bibitem{AldousFill}
D. Aldous and J. Fill,
\emph{Reversible Markov Chains and Random Walks on Graphs},
unfinished monograph, available online at \url{https://www.stat.berkeley.edu/~aldous/RWG/book.html}.

\bibitem{KemenySnell}
J. G. Kemeny and J. L. Snell,
\emph{Finite Markov Chains},
Springer, New York, 1976.

\bibitem{SimonAndo}
H. A. Simon and A. Ando,
Aggregation of variables in dynamic systems,
\emph{Econometrica} \textbf{29} (1961), 111--138.

\bibitem{Meyer}
C. D. Meyer,
Stochastic complementation, uncoupling Markov chains, and the theory of nearly reducible systems,
\emph{SIAM Review} \textbf{31} (1989), 240--272.

\bibitem{DuanDunsonCarin}
L. Duan, D. B. Dunson, and L. Carin,
Spectral state compression of Markov processes,
\emph{Adv. Neural Inf. Process. Syst.} \textbf{32} (2019), 7586--7595.

\end{thebibliography}
\end{document}